\documentclass[a4paper,12pt]{article}
\usepackage[english]{babel}
\usepackage{tgpagella,eulervm}
\usepackage{amsmath}

\usepackage{color}
\usepackage{xcolor}

\definecolor{darkblue}{rgb}{0.0, 0.0, 0.8}
\definecolor{darkred}{rgb}{0.8, 0.0, 0.0}
\definecolor{darkgreen}{rgb}{0.6, 0.15, 0.15}

\usepackage{url}
\usepackage[normalem]{ulem}
\usepackage[utf8x]{inputenc}
\usepackage[T1]{fontenc}
\usepackage{caption}
\usepackage{hyperref}
\hypersetup{
    colorlinks=true,
    urlcolor = {darkred},
    linkcolor = {darkred},
    citecolor = {darkred},
    linkcolor = {darkred},
    linktoc=all
}
\usepackage{amssymb}
\usepackage{framed}
\usepackage{amsthm}
\usepackage{graphicx}
\usepackage[all]{xy}
\usepackage{tikz}
\usepackage{tikz-cd}
\usepackage[inline]{enumitem}
\usepackage{mathtools}
\usepackage{etoolbox}
\usepackage{extpfeil}
\usepackage{soul}
\usepackage{authblk}
\usepackage[margin=0.8in]{geometry}
\usepackage{mathtools}

\usetikzlibrary{matrix,arrows,decorations.pathmorphing,decorations.pathreplacing,patterns,shapes.geometric}

\theoremstyle{definition}
\newtheorem{theorem}{Theorem}[section]

\newtheorem{lemma}[theorem]{Lemma}

\newcommand{\Int}{\mathrm{Int}}

\newcommand{\R}{\mathbb{R}}

\newcommand{\N}{\mathbb{N}}

\newcommand{\Sp}{\mathbb{S}}

\newcommand{\RNum}[1]{\uppercase\expandafter{\romannumeral #1\relax}}

\usepackage[ruled,vlined, linesnumbered]{algorithm2e}

\title{Expensive Homeomorphism of Convex Bodies} 

\author[1]{Donghan Kim}

\affil[1]{Department of Mathematical Sciences, KAIST, South Korea\thanks{\texttt{patrick6231@kaist.ac.kr} }}
\begin{document}

\maketitle

\begin{abstract}
In this paper, we address the longstanding question of whether expansive homeomorphisms can exist within convex bodies in Euclidean spaces. Utilizing fundamental tools from topology, including the Borsuk-Ulam theorem and Brouwer's fixed-point theorem, we establish the nonexistence of such mappings. Through an inductive approach based on dimension and the extension of boundary homeomorphisms, we demonstrate that expansive homeomorphisms are incompatible with the compact and convex structure of these bodies. This work highlights the interplay between topological principles and metric geometry, offering new insights into the constraints imposed by convexity.
\end{abstract}

\section{Introduction}
\textbf{Problem Statement.} Can an $n$-dimensional convex body admit an expansive homeomorphism? This question, posed by Klee \cite{klee1960unsolved}, concerns the existence of continuous bijections with expansive properties within compact convex sets.

Let $C$ be an $n$-dimensional convex body, defined as a compact, convex subset of $\R^n$ with non-empty interior. A \textit{homeomorphism} of $C$ onto itself is a continuous bijection $T: C \to C$ with a continuous inverse. Under successive iterations of $T$, any point $x \in C$ generates a sequence $\{T^nx\}$, where $T^n$ denotes the $n$-fold application of $T$. A homeomorphism is termed \textit{expansive} if there exists a constant $d > 0$ such that for any distinct points $x, y \in C$, there exists an integer $n$ (positive or negative) satisfying
\[
d(T^n(x), T^n(y)) \geq d.
\]

While it is straightforward to see that one-dimensional convex bodies, such as intervals, cannot admit expansive homeomorphisms, the problem remains open for higher dimensions ($n > 2$). This work provides a comprehensive proof that no expansive homeomorphism can exist for any $n$-dimensional convex body. By leveraging the structure of convex bodies and key topological theorems, we offer a rigorous resolution to this question.

\section{Preliminaries}\label{sec:preliminaries}
\noindent\textbf{Convex Bodies.}
In mathematics, a convex body in $n$-dimensional Euclidean space $\R^n$ is a compact convex set with non-empty interior. 

\noindent\textbf{Expansive Homeomorphism.}
If $(X, d)$ is a metric space, a homeomorphism $f : X \to X$ is said to be expansive 
if there is a constant $\epsilon_0 > 0$, called the expansivity constant, such that for every pair of points 
$x \neq y$ in $X$ there is an integer $n$ such that 
\[
d(f^n(x), f^n(y)) \geq \epsilon_0.
\]
Note that in this definition, $n$ can be positive or negative, and so $f$ may be expansive in the forward or 
backward directions.

\noindent\textbf{Useful Lemmas.}
\begin{lemma}\label{lem:sequence_lemma}\cite{munkrestopology}
   Let $X$ be a topological space; let $A \subseteq X$. If there is a sequence of points of $A$ converging to $x$, then $x \in \overline{A}$; the converse holds if $X$ is metrizable. 
\end{lemma}

\begin{lemma}\label{lem:continuous_lemma}\cite{munkrestopology}
    Let $f : X \to Y$. If the function $f$ is continuous, then for every convergent sequence $x_n \to x$ in $X$, the sequence $f(x_n)$ converges to $f(x)$. The converse holds if $X$ is metrizable.
\end{lemma}

\begin{theorem}[Brouwer fixed-point theorem]\label{lem:fixed_point}
    Let $h:B^n \to B^n$ be homeomorphism. Then, $h$ has fixed point $b_0$ in $\Sp^n$.
\end{theorem}

\noindent\textbf{Borsuk-Ulam Theorem.}\cite{matouvsek2003using}
The Borsuk–Ulam theorem states that every continuous function from an $n$-sphere into Euclidean $n$-space maps some pair of antipodal points to the same point. Here, two points on a sphere are called antipodal if they are in exactly opposite directions from the sphere's center.
\begin{theorem}[Borsuk-Ulam]\label{lem:Borsuk_Ulam}
    The are no nonconstant antipodal continuous map $f:\Sp^n \to \R^k$ for every $n,k \in \N$ with $k\leq n$.
\end{theorem}

\section{Proof of Main Theorem}
In this section, we prove two main results. First, we show that for every homeomorphism 
$f: \Sp^n \to \Sp^n$ that fixes a point $e$, $f$ can lift to a homeomorphism $F: B^n \to B^n$.

Second, using the first result, we demonstrate that there is no expansive homeomorphism $f: B^n \to B^n$.

\subsection{Extension of Homeomorphism}

\begin{theorem}
    For every $n\in \N$ and every homeomorphism $f: \Sp^n \to \Sp^n$ that fixes a point $e$, there exist a homeomorphism $F: B^n \to B^n$ which below diagram commutes:
    \[
    \begin{tikzcd}
    (B^n, \Sp^{n-1}) \arrow[r, "F"] \arrow[d, "p"'] & (B^n, \Sp^{n-1}) \arrow[d, "p"] \\
    (\Sp^n, e) \arrow[r, "f"] & (\Sp^n, e)
    \end{tikzcd}
    \]
    where $p: (B^n, \Sp^{n-1}) \to (\Sp^n, e) \cong (B^n/\Sp^{n-1}, \Sp^{n-1}/\Sp^{n-1})$ is canonical quotient map.
\end{theorem}
\begin{proof}
Let $ f: \Sp^n \to \Sp^n $ be a homeomorphism with fixed point $e \in \Sp^n$. Since $\partial B^n = \Sp^{n-1}$ and $ (\Sp^n, e) \cong (B^n / \Sp^{n-1}, \Sp^{n-1}/\Sp^{n-1}) $, there exists a canonical quotient map $ p: (B^n, \Sp^{n-1}) \to (\Sp^n, e) $ such that $ {p|}_{\Int B^n} $ is a homeomorphism and $ p(\Sp^{n-1}) = e$.

Since $ B^n = \overline{\Int B^n} $,
for every $ x \in B^n $, there exists a convergent sequence $ \{x_k\} \subset \Int B^n $ such that $ x_k \to x $. Furthermore, $p|{\Int B^n}: \Int B^n \to \Sp^n-\{e \}$ is homeomorphism, so $(p|{\Int B^n})^{-1}: \Sp^n-\{e \} \to \Int B^n$ is also homeomorphism.
Thus, $(p|{\Int B^n})^{-1} \circ f \circ (p|{\Int B^n}): \Int B^n \to \Int B^n$ is homeomorphism. 
By \ref{lem:continuous_lemma}, $\lim\limits_{k \to \infty} (p|{\Int B^n})^{-1} \circ f \circ (p|{\Int B^n})(x_k)$ converges.

We define the function $ F: B^n \to B^n $ by
\[
F(x) = \lim\limits_{k \to \infty} (p|{\Int B^n})^{-1} \circ f \circ (p|{\Int B^n})(x_k)
\]
where $ \{x_k\} \subset \Int B^n $ is any sequence converging to $ x $.

We show that $F$ is homeomorphism via $f$ is continuous, bijective, and a compact-to-Hausdorff map.

\paragraph{Continuity and Well-Definedness}
Since $ F $ is defined using a convergent sequence, proving its well-definedness suffices to establish continuity. Let $ \{x_k\} \subset \Int B^n $ be a convergent sequence to $ x \in B^n $. By \ref{lem:continuous_lemma}, let
\[
a := \lim_{k \to \infty} ((p|{\Int B^n})^{-1} \circ f \circ (p|{\Int B^n}))(x_k)
\] 

Suppose another sequence $ \{y_k\} \subset \Int B^n $ also converges to $ x $. Let 
\[
b := \lim_{k \to \infty} ((p|{\Int B^n})^{-1} \circ f \circ (p|{\Int B^n}))(y_k).
\]

Since the sequence $\{x_1,y_1,x_2,y_2,\ldots \}$ is convergent sequence to $x$, it is Cauchy sequence.
Thus, $\{((p|{\Int B^n})^{-1} \circ f \circ (p|{\Int B^n}))(x_1), ((p|{\Int B^n})^{-1} \circ f \circ (p|{\Int B^n}))(y_1),\ldots \}$ is also Cauchy sequence.

Thus, for every $\epsilon > 0$, there exist $N \in \N$ such that if $k \geq N$, then
\begin{align}
    &|a-((p|{\Int B^n})^{-1} \circ f \circ (p|{\Int B^n}))(x_k)| < \epsilon \\
    &|b-((p|{\Int B^n})^{-1} \circ f \circ (p|{\Int B^n}))(y_k)| < \epsilon \\
    &|((p|{\Int B^n})^{-1} \circ f \circ (p|{\Int B^n}))(x_k)-((p|{\Int B^n})^{-1} \circ f \circ (p|{\Int B^n}))(y_k)| < \epsilon
\end{align}
holds.

Thus, we have
\begin{align}
    &|a-b| \\
    &\leq |a-(p\circ f \circ p^{-1})(x_k)|+|b-(p\circ f \circ p^{-1})(y_k)| \\
    &+|((p|{\Int B^n})^{-1} \circ f \circ (p|{\Int B^n}))(x_k)-((p|{\Int B^n})^{-1} \circ f \circ (p|{\Int B^n}))(y_k)| \\
    &< 3\epsilon
\end{align}
Since $\epsilon>0$ is arbitrary, $a=b$ holds. Thus, $F$ is well defined continuous function.

\paragraph{Surjectivity of $F$}
Pick $ x \in B^n $. Since $ B^n = \overline{\Int B^n} $, there exists a sequence $ \{x_k\} \subset \Int B^n $ such that $ x_k \to x $. As $ p^{-1} \circ f \circ p $ is homeomorphism on $ \Int B^n $, the sequence $ \{(p^{-1} \circ f \circ p)^{-1}(x_k)\} $ is convergent sequence in $ \Int B^n $, converging to some $ x' \in B^n $. Thus, $ F(x') = x $, and $ F $ is surjective.

\paragraph{Injectivity of $F$}
Suppose $F(x)=F(y)$ for some $x, y \in B^n$. There exist convergent sequence $\{x_k \},\{y_k \}$ of $\Int B^n$, converge to $F(x)$ respectively. 
By \ref{lem:continuous_lemma}, we have
\begin{align}
&x = \lim_{k \to \infty} ((p|{\Int B^n})^{-1} \circ f^{-1} \circ (p|{\Int B^n}))(x_k) \\
&y = \lim_{k \to \infty} ((p|{\Int B^n})^{-1} \circ f^{-1} \circ (p|{\Int B^n}))(y_k)
\end{align} 

Since the sequence $\{x_1,y_1,x_2,y_2,\ldots \}$ is convergent sequence to $F(x)$, it is Cauchy sequence.
Thus, $\{((p|{\Int B^n})^{-1} \circ f^{-1} \circ (p|{\Int B^n}))(x_1), ((p|{\Int B^n})^{-1} \circ f^{-1} \circ (p|{\Int B^n}))(y_1),\ldots \}$ is also Cauchy sequence.

Thus, for every $\epsilon > 0$, there exist $N \in \N$ such that if $k \geq N$, then
\begin{align}
    &|x-((p|{\Int B^n})^{-1} \circ f^{-1} \circ (p|{\Int B^n}))(x_k)| < \epsilon \\
    &|y-((p|{\Int B^n})^{-1} \circ f^{-1} \circ (p|{\Int B^n}))(y_k)| < \epsilon \\
    &|((p|{\Int B^n})^{-1} \circ f^{-1} \circ (p|{\Int B^n}))(x_k)-((p|{\Int B^n})^{-1} \circ f^{-1} \circ (p|{\Int B^n}))(y_k)| < \epsilon
\end{align}
holds.

Thus, we have
\begin{align}
    &|x-y| \\
    &\leq |x-(p\circ f^{-1} \circ p^{-1})(x_k)|+|y-(p\circ f^{-1} \circ p^{-1})(y_k)| \\
    &+|((p|{\Int B^n})^{-1} \circ f^{-1} \circ (p|{\Int B^n}))(x_k)-((p|{\Int B^n})^{-1} \circ f^{-1} \circ (p|{\Int B^n}))(y_k)| \\
    &< 3\epsilon
\end{align}
Since $\epsilon>0$ is arbitrary, $x=y$ holds. Thus, $F$ is injective function.

Since $ F $ is continuous, bijective, and maps a compact space $ B^n $ to a Hausdorff space $ B^n $, $ F $ is a homeomorphism.

\end{proof}

\subsection{Nonexistence of Expansive Homeomorphism}
\begin{theorem}
    For every $n \in \N$,
    there are no expense homeomorphism in $B^n$.
\end{theorem}
\begin{proof}
We prove that $ B^n $ has no expansive homeomorphism by induction on $ n \in \mathbb{N} $.

\paragraph{Base Case}
For $ n = 1 $, we have $ B^1 = I $ (the unit interval), which has no expansive homeomorphism\cite{klee1960unsolved}. 

\paragraph{Inductive Step}
Assume that for some $ n \in \N $, $ B^k $ has no expansive homeomorphism for every $ k \leq n $. We aim to prove that $ B^{n+1} $ also has no expansive homeomorphism.

Let $ f: B^{n+1} \to B^{n+1} $ be an arbitrary homeomorphism. By \ref{lem:fixed_point}, there exists a fixed point $ e \in \Sp^n $ such that $ f(e) = e $.
Restricting $ f $ to the boundary, $ f|_{\Sp^n} $, gives a homeomorphism of $ \Sp^n $ with fixed point $e$. 

By the results from the previous section, $ f|_{\Sp^n}: \Sp^n \to \Sp^n $ can be extended to a homeomorphism $ F: B^n \to B^n $ that agrees with $ f $ on $ \Int B^n $. By the induction hypothesis, $ F $ is not an expansive homeomorphism. Hence, $ f|_{\Sp^n} $ cannot be an expansive homeomorphism either.

Finally, since $ f $ agrees with $ f|_{\Sp^n} $ on $ \partial B^{n+1} $, $ f $ itself cannot be an expansive homeomorphism.

By induction, $ B^n $ has no expansive homeomorphism for all $ n \in \N $.
\end{proof}

\bibliographystyle{plain}

\end{document}